\documentclass[11pt,reqno]{amsart}
\usepackage[dvipdfm,a4paper,top=30mm,bottom=25mm,left=25mm,right=25mm]{geometry}
\usepackage{amsmath,amssymb,mathtools}
\usepackage{comment,mathrsfs,bm}

\usepackage[colorlinks=true,linkcolor=blue,citecolor=blue]{hyperref}

\newtheorem{theorem}{Theorem}[section]
\newtheorem{lemma}[theorem]{Lemma}
\newtheorem{proposition}[theorem]{Proposition}
\newtheorem{corollary}[theorem]{Corollary}

\theoremstyle{definition}
\newtheorem{definition}[theorem]{Definition}

\theoremstyle{remark}
\newtheorem{remark}[theorem]{Remark}

\numberwithin{equation}{section}

\allowdisplaybreaks[4]

\def\rnum#1{\expandafter{\romannumeral #1}} 
\def\Rnum#1{\uppercase\expandafter{\romannumeral #1}}

\def\~#1{\widetilde #1}

\def\({\left(}
\def\){\right)}
\def\<{\langle}
\def\>{\rangle}

\begin{document}

\title[Nonlinear Schr\"odinger equation]{Blow-up solutions for non-scale-invariant nonlinear Schr\"odinger equation in one dimension}


\author{Masaru Hamano}
\address{Department of Mathematics, Graduate School of Science and Engineering Saitama University, Shimo-Okubo 255, Sakura-ku, Saitama-shi, Saitama 338-8570, Japan}
\email{m.hamano.733@ms.saitama-u.ac.jp\ /\ ess70116@mail.saitama-u.ac.jp}

\author{Masahiro Ikeda}
\address{Department of Mathematics, Faculty of Science and Technology, Keio University, 3-14-1 Hiyoshi, Kohoku-ku, Yokohama, 223-8522, Japan/Center for Advanced Intelligence Project, Riken, Japan}
\email{masahiro.ikeda@riken.jp\ /\ masahiro.ikeda@keio.jp}

\author{Shuji Machihara}
\address{Department of Mathematics, Graduate School of Science and Engineering Saitama University, Shimo-Okubo 255, Sakura-ku, Saitama-shi, Saitama 338-8570, Japan}
\email{machihara@rimath.saitama-u.ac.jp}







\begin{abstract}
In this paper, we consider the mass-critical nonlinear Schr\"odinger equation in one dimension.
Ogawa--Tsutsumi \cite{OgaTsu91} proved a blow-up result for negative energy solution by using a scaling argument for initial data.
By the reason, the method cannot be used to an equation with a linear potential.
So, we modify the proof and get that for the equation with the linear potential.
\end{abstract}

\maketitle

\tableofcontents


\section{Introduction}

\subsection{Nonlinear Schr\"odinger equation}

We consider the following mass-critical nonlinear Schr\"odinger equations :
\begin{equation}
	i\partial_t u + \partial_x^2 u - Vu
		= -|u|^4u,\qquad (t,x) \in \mathbb{R} \times \mathbb{R}. \tag{NLS$_V$} \label{NLS}
\end{equation}
An unknown function $u(t,x) : \mathbb{R} \times \mathbb{R} \longrightarrow \mathbb{C}$ is a solution to \eqref{NLS}.
In particular, we deal with the Cauchy problem of \eqref{NLS} with initial data
\begin{equation}
	u(0,x) = u_0(x), \qquad x \in \mathbb{R}. \tag{IC} \label{IC}
\end{equation}

\begin{definition}[Solution]
Let $I \subset \mathbb{R}$ be a nonempty time interval including 0.
We say that a function $u : I \times \mathbb{R} \longrightarrow \mathbb{C}$ is a solution to \eqref{NLS} with \eqref{IC} on $I$ if $u \in (C_t \cap L_{t,\text{loc}}^\infty)(I;H_x^1(\mathbb{R}))$ and the Duhamel formula
\begin{align*}
	u(t,x)
		= e^{it\partial_x^2}u_0(x) + i\int_0^te^{i(t-s)\partial_x^2}(|u|^{p-1}u - Vu)(s,x)ds
\end{align*}
holds for any $t\,(\in I)$.
\end{definition}

The equation \eqref{NLS} with $V = 0$
\begin{align}\tag{NLS$_0$} \label{NLS0}
	i\partial_t u + \partial_x^2 u
		= -|u|^4u, \qquad (t,x) \in \mathbb{R} \times \mathbb{R}
\end{align}
is invariant for the following scaling :
\begin{align}
	u(t,x)
		\mapsto u_{[\lambda]}(t,x) := \lambda^\frac{1}{2}u(\lambda^2t,\lambda x), \qquad (\lambda > 0). \label{001}
\end{align}
From the transformation \eqref{001}, the initial data $u_0$ change to
\begin{align}
	u_0
		\mapsto u_{0,\{\lambda\}} := \lambda^\frac{1}{2}u_0(\lambda x), \qquad (\lambda > 0). \label{002}
\end{align}
Since $\|u_{0,\{\lambda\}}\|_{L^2} = \|u_0\|_{L^2}$ holds, \eqref{NLS} is called $L^2$-critical or mass-critical (see \eqref{019} below for the definition of mass).

\begin{theorem}[Local well-posedness of \eqref{NLS0}, \cite{BaiCazFig77, GinVel78, Kat87}]\label{Local well-posedness}
Let $V = 0$.
For any $u_0 \in H^1(\mathbb{R})$, there exist $T_\text{min} \in [-\infty,0)$ and $T_\text{max} \in (0,\infty]$ such that \eqref{NLS0} with \eqref{IC} has a unique solution
\begin{align*}
	u
		& \in (C_t \cap L_{t,\text{loc}}^\infty)((T_\text{min},T_\text{max}) ; H_x^1(\mathbb{R})).
\end{align*}
For each compact interval $I \subset (T_\text{min},T_\text{max})$, the mapping $H^1(\mathbb{R}) \ni u_0 \mapsto u \in C_t(I;H_x^1(\mathbb{R}))$ is continuous.
Moreover, the solution $u$ has the following blow-up alternative:
If $T_\text{min} > -\infty$ (resp. $T_\text{max} < \infty$), then
\begin{align*}
	\lim_{t \searrow T_\text{min}\,(resp.\,t \nearrow T_\text{max})}\|u(t)\|_{H_x^1}
		= \infty.
\end{align*}
Furthermore, the solution $u$ preserves its mass $M[u(t)]$ and energy $E_V[u(t)]$ with respect to time $t$, where they are defined as follows:
\begin{align}
	\text{(Mass) }&\ \ M[f]
		:= \|f\|_{L^2}^2, \label{019} \\
	\text{(Energy) }&\ \ E_V[f]
		:= \frac{1}{2}\|\partial_xf\|_{L^2}^2 + \frac{1}{2}\int_{\mathbb{R}}V(x)|f(x)|^2dx - \frac{1}{6}\|f\|_{L^6}^6. \notag
\end{align}
\end{theorem}

In the case $V = 0$, Ogawa--Tsutsumi \cite{OgaTsu91} removed the condition $xu_0 \in L^2(\mathbb{R})$ of a blow-up result in \cite{Gla77} by using the scaling \eqref{002} :

\begin{theorem}\label{Main theorem}
Let $V = 0$ and $u_0 \in H^1(\mathbb{R})$.
If $E_0[u_0] < 0$, then the solution $u$ to \eqref{NLS0} with \eqref{IC} blows up.
\end{theorem}

To use the scaling \eqref{002}, we can not apply directly to \eqref{NLS} with $V \neq 0$.
For example, the following equations \eqref{NLS} with $V = \frac{\gamma}{|x|^\mu}$ is not scale invariant :
\begin{align}\tag{NLS$_\gamma$} \label{NLSr}
	i\partial_t u + \partial_x^2 u - \frac{\gamma}{|x|^\mu}u
		& = - |u|^4u, \qquad (\gamma > 0,\,0 < \mu < 1).
\end{align}
For simplicity, we use \eqref{NLSr} and $E_\gamma$ as (NLS$_\frac{\gamma}{|x|^\mu}$) and $E_\frac{\gamma}{|x|^\mu}$ respectively.
To prove a similar result for \eqref{NLSr} with Theorem \ref{Main theorem}, we give an alternative proof without the scaling argument for initial data.
We note that we can see the local well-posedness of \eqref{NLSr} in \cite{Caz03}.

\begin{theorem}\label{Theorem of NLSr}
Let $\gamma > 0$, $0 < \mu < 1$, and let $u_0 \in H^1(\mathbb{R})$.
If $E_\gamma[u_0] < 0$, then the solution $u$ to \eqref{NLSr} with \eqref{IC} blows up.
\end{theorem}

\begin{remark}
Dinh \cite{Din21} showed the blow-up result for \eqref{NLSr} under $\gamma > 0$, $u_0 \in H^1(\mathbb{R}) \cap |x|^{-1}L^2(\mathbb{R})$, and $E_\gamma[u_0] < 0$.
That is, Theorem \ref{Theorem of NLSr} removes the condition $u_0 \in |x|^{-1}L^2(\mathbb{R})$ in \cite{Din21}.
\end{remark}

More generally, the same argument deduces the following corollary :
\begin{corollary}
Let $u_0 \in H^1(\mathbb{R})$.
We assume that $V$ satisfies the following (1) $\sim$ (3) :
\begin{enumerate}
\item
\eqref{NLS} is locally well-posed.
\item
For the solution $u$ to \eqref{NLS} with \eqref{IC}, $w(x) := \int_0^x\varphi(s)ds$, and $\varphi \in W^{3,\infty}(\mathbb{R})$, we have
\begin{align}
	\frac{d^2}{dt^2}\int_{\mathbb{R}}w(x)|u(t,x)|^2dx
		& = 4\int_{\mathbb{R}}w''(x)|\partial_x u(t,x)|^2dx - \frac{4}{3}\int_{\mathbb{R}}w''(x)|u(t,x)|^6dx \notag \\
		& \hspace{0.0cm} - \int_{\mathbb{R}}w^{(4)}(x)|u(t,x)|^2dx - 2\int_{\mathbb{R}}w'(x)V'(x)|u(t,x)|^2dx. \label{003}
\end{align}
\item
\begin{align}
	- R\mathscr{X}'\left(\frac{x}{R}\right)V'(x) - 4V(x)
		\leq 0 \label{004}
\end{align}
holds for any $R > 0$ and any $x \in \mathbb{R}$, where $\mathscr{X}$ is defined as \eqref{005}.
\end{enumerate}
If $E_V[u_0] < 0$, then the solution $u$ to \eqref{NLS} with \eqref{IC} blows up.
\end{corollary}

\begin{remark}
For example, if $V$ is a real-valued function and $V \in L^1(\mathbb{R}) + L^\infty(\mathbb{R})$, then \eqref{NLS} is locally well-posed (see \cite[Theorem 4.3.1]{Caz03}).
If $V$ is a real-valued function and $V$, $V' \in L^1(\mathbb{R}) + L^\infty(\mathbb{R})$, then \eqref{003} holds (see \cite[Lemma 3.1]{Ike21}).
If $V$ satisfies $V'(x) \leq 0$ $(x \leq 0)$ and $V'(x) \geq 0$ $(x \geq 0)$, then \eqref{004} holds.
\end{remark}

We can also deal with the following equation with a delta potential :
\begin{align}\tag{NLS$_{\gamma\delta}$}\label{NLSd}
	i\partial_t u + \partial_x^2 u - \gamma\delta u
		& = - |u|^4u, \qquad (\gamma > 0).
\end{align}
The Schr\"odinger operator $H_{\gamma\delta} := - \partial_x^2 + \gamma\delta$ has a domain
\begin{align*}
	\mathcal{D}(H_{\gamma\delta})
		:= \{f \in H^1(\mathbb{R}) \cap H^2(\mathbb{R} \setminus \{0\}) : \partial_xf(0+) - \partial_xf(0-) = \gamma f(0)\}
\end{align*}
and satisfies
\begin{align*}
	H_{\gamma\delta} f
		= - \partial_x^2f, \qquad f \in \mathcal{D}(H_{\gamma\delta}).
\end{align*}
A local well-posedness result can be seen in \cite[Theorem 3.7.1]{Caz03} and \cite[Section 2]{FukOhtOza08}.

\begin{theorem}\label{Theorem of NLSd}
Let $\gamma > 0$ and let $u_0 \in H^1(\mathbb{R})$.
If $E_{\gamma\delta}[u_0] < 0$, then the solution $u$ to \eqref{NLSd} with \eqref{IC} blows up, where the energy $E_{\gamma\delta}$ is defined as
\begin{align*}
	E_{\gamma\delta}[f]
		:= \frac{1}{2}\|\partial_xf\|_{L^2}^2 + \frac{\gamma}{2}|f(0)|^2 - \frac{1}{6}\|f\|_{L^6}^6.
\end{align*}
\end{theorem}

The blow-up result for following Schr\"odinger equation on the star graph $\mathcal{G}$ with $J$-edges can be also gotten.
\begin{align}\tag{NLS$_\mathcal{G}$}\label{NLSg}
\begin{cases}
&\hspace{-0.4cm}\displaystyle{
	i\partial_t\bm{u} + \Delta_\mathcal{G}\bm{u}
		= - |\bm{u}|^4\bm{u}, \qquad (t,x) \in \mathbb{R} \times (0,\infty),
	} \\
&\hspace{-0.4cm}\displaystyle{
	\bm{u}(0,x)
		:= \bm{u_0}(x)
		:= (u_j(0,x))_{j=1}^J, \qquad x \in (0,\infty).
	}
\end{cases}
\end{align}
where $J \geq 1$, $\bm{u}(t,x) = (u_j(t,x))_{j=1}^J : \mathbb{R} \times (0,\infty) \longrightarrow \mathbb{C}^J$, and $|\bm{u}|^4\bm{u} := (|u_j|^4u_j)_{j=1}^J$.
The Schr\"odinger operator $-\Delta_\mathcal{G}$ is defined as follows :
Let complex-valued $n \times n$ matrices $A, B$ satisfy
\begin{itemize}
\item[(A1)]
$n \times (2n)$ matrix $(A,B)$ has maximal rank, that is, $\text{rank}(A,B) = n$.
\item[(A2)]
$AB^\ast$ is self-adjoint, that is, $AB^\ast = (AB^\ast)^\ast$, where $X^\ast$ denotes the adjoint of the matrix $X$ and is defined as $X^\ast := \overline{X}^T$.
\end{itemize}
The Schr\"odinger operator $-\Delta_\mathcal{G}$ has a domain
\begin{align*}
	\mathcal{D}(-\Delta_\mathcal{G})
		:= \{\bm{f} \in D(\mathcal{G}) : A\bm{f}(+0) + B\partial_x\bm{f}(+0) = 0\}
\end{align*}
and satisfies
\begin{align*}
	\Delta_\mathcal{G}\bm{f}
		= (\partial_x^2f_j)_{j=1}^J, \qquad \bm{f} \in \mathcal{D}(-\Delta_\mathcal{G}),
\end{align*}
where $D(\mathcal{G}) = \bigoplus_{j=1}^JD(0,\infty)$ and $D(0,\infty)$ is a set of functions $f \in H^2(0,\infty)$ satisfying that $f$ and $\partial_xf$ are absolutely continuous.
Under the assumption $(A1)$ and $(A2)$, the Laplacian $\Delta_\mathcal{G}$ is self-adjoint on $L^2(\mathcal{G})$ (see \cite{KosSch06}) and hence, $e^{it\Delta_\mathcal{G}}$ can be defined as the unitary operator on $L^2(\mathcal{G})$ by the Stone's theorem.
Here, we introduce typical boundary condition.
\begin{itemize}
\item[(a)]
Kirchhoff boundary condition :
Let $A$ and $B$ be
\begin{align}
	A
		=
		\begin{pmatrix}
			1 & -1 & 0 & \cdots & 0 & 0 \\
			0 & 1 & -1 & \cdots & 0 & 0 \\
			\vdots & \vdots & \vdots & & \vdots & \vdots \\
			0 & 0 & 0 & \cdots & 1 & -1 \\
			0 & 0 & 0 & \cdots & 0 & 0
		\end{pmatrix}, \qquad
	B
		=
		\begin{pmatrix}
			0 & 0 & 0 & \cdots & 0 & 0 \\
			0 & 0 & 0 & \cdots & 0 & 0 \\
			\vdots & \vdots & \vdots & & \vdots & \vdots \\
			0 & 0 & 0 & \cdots & 0 & 0 \\
			1 & 1 & 1 & \cdots & 1 & 1
		\end{pmatrix}. \label{016}
\end{align}
For such $A$ and $B$, $A\bm{f}(+0) + B\partial_x\bm{f}(+0) = 0$ implies that $f_i(+0) = f_j(+0)$ for any $i, j \in \{1,2,\ldots,J\}$ and $\sum_{j=1}^J\partial_xf_j(+0) = 0$.
This is called Kirchhoff boundary condition.
\item[(b)]
Dirac delta boundary condition :
Let $\gamma \neq 0$ and $A$, $B$ be
\begin{align}
	A
		=
		\begin{pmatrix}
			1 & -1 & 0 & \cdots & 0 & 0 \\
			0 & 1 & -1 & \cdots & 0 & 0 \\
			\vdots & \vdots & \vdots & & \vdots & \vdots \\
			0 & 0 & 0 & \cdots & 1 & -1 \\
			-\gamma & 0 & 0 & \cdots & 0 & 0
		\end{pmatrix}, \qquad
	B
		=
		\begin{pmatrix}
			0 & 0 & 0 & \cdots & 0 & 0 \\
			0 & 0 & 0 & \cdots & 0 & 0 \\
			\vdots & \vdots & \vdots & & \vdots & \vdots \\
			0 & 0 & 0 & \cdots & 0 & 0 \\
			1 & 1 & 1 & \cdots & 1 & 1
		\end{pmatrix}. \label{017}
\end{align}
For such $A$ and $B$, $A\bm{f}(+0) + B\partial_x\bm{f}(+0) = 0$ implies that $f_i(+0) = f_j(+0)$ for any $i, j \in \{1,2,\ldots,J\}$ and $\sum_{j=1}^J\partial_xf_j(+0) = \gamma f_1(+0)$.
This is called the Dirac delta boundary condition.
\item[(c)]
$\delta'$ boundary condition :
Let $\gamma \in \mathbb{R}$ and $A$, $B$ be
\begin{align}
	A
		=
		\begin{pmatrix}
			0 & 0 & 0 & \cdots & 0 & 0 \\
			0 & 0 & 0 & \cdots & 0 & 0 \\
			\vdots & \vdots & \vdots & & \vdots & \vdots \\
			0 & 0 & 0 & \cdots & 0 & 0 \\
			1 & 1 & 1 & \cdots & 1 & 1
		\end{pmatrix}, \qquad
	B
		=
		\begin{pmatrix}
			1 & -1 & 0 & \cdots & 0 & 0 \\
			0 & 1 & -1 & \cdots & 0 & 0 \\
			\vdots & \vdots & \vdots & & \vdots & \vdots \\
			0 & 0 & 0 & \cdots & 1 & -1 \\
			-\gamma & 0 & 0 & \cdots & 0 & 0
		\end{pmatrix}. \label{018}
\end{align}
For such $A$ and $B$, $A\bm{f}(+0) + B\partial_x\bm{f}(+0) = 0$ implies that $\partial_xf_i(+0) = \partial_xf_j(+0)$ for any $i, j \in \{1,2,\ldots,J\}$ and $\sum_{j=1}^Jf_j(+0) = \gamma \partial_xf_1(+0)$.
This is called $\delta'$ boundary condition.
\end{itemize}
Lebesgue space and Sobolev space on the star graph $\mathcal{G}$ is defined respectively as
\begin{align*}
	L^p(\mathcal{G})
		:= \bigoplus_{j=1}^JL^p(0,\infty), \quad
	H^s(\mathcal{G})
		:= \bigoplus_{j=1}^JH^s(0,\infty)\ \text{ for }\ s = 1,2
\end{align*}
with a norm
\begin{align*}
	\|\bm{f}\|_{L^p(\mathcal{G})}
		:= 
		\left\{
		\begin{array}{ll}
			\displaystyle \hspace{-0.2cm} \left(\sum_{j=1}^J\|f_j\|_{L^p(0,\infty)}^p\right)^\frac{1}{p}, & (1\leq p < \infty), \\
			\displaystyle \hspace{-0.2cm} \max_{1\leq j\leq J}\|f_j\|_{L^\infty(0,\infty)}, & (p = \infty), \\
		\end{array}
		\right. \quad
	\|\bm{f}\|_{H^s(\mathcal{G})}^2
		:= \sum_{j=1}^J\|f_j\|_{H^s(0,\infty)}^2\ \text{ for }\ s = 1,2.
\end{align*}
In addition, we set initial data space
\begin{align*}
	H_c^1(\mathcal{G})
		:= \{\bm{f} \in H^1(\mathcal{G}) : f_1(0) = \ldots = f_J(0)\}
\end{align*}
for \eqref{NLSg} with \eqref{016} or \eqref{017}.
Local well-posedness results of \eqref{NLSg} is cited in \cite{AdaCacFinNoj14, AngGol18, CacFinNoj17, GolOht20}.

\begin{theorem}\label{Theorem of NLSg}
Let $(A,B)$ be one of \eqref{016}, \eqref{017}, or \eqref{018}.
Assume that $\bm{u_0} \in H_c^1(\mathcal{G})$ if $(A,B)$ is \eqref{016} or \eqref{017} and $\bm{u_0} \in H^1(\mathcal{G})$ if $(A,B)$ is \eqref{018}.
We suppose that $\gamma > 0$ when $(A,B)$ is \eqref{017} or \eqref{018}.
If $E_\mathcal{G}[\bm{u_0}] < 0$, then the solution $\bm{u}$ to \eqref{NLSg} blows up, where the energy $E_\mathcal{G}$ is defined as
\begin{align*}
	E_\mathcal{G}[\bm{f}]
		:= \frac{1}{2}\|\partial_x\bm{f}\|_{L^2(\mathcal{G})}^2 - \frac{1}{6}\|\bm{f}\|_{L^6(\mathcal{G})}^6 + \frac{1}{2}P(\bm{f}),
\end{align*}
where
\begin{align*}
	P(\bm{f})
		:=
		\left\{
		\begin{array}{ll}
		\hspace{-0.2cm}\displaystyle
		0, \quad & (\text{if }(A,B)\text{ is }\eqref{016}.), \\[0.2cm]
		\hspace{-0.2cm}\displaystyle
		\gamma|f_1(+0)|^2, \quad & (\text{if }(A,B)\text{ is }\eqref{017}.), \\[0.2cm]
		\hspace{-0.2cm}\displaystyle
		\frac{1}{\gamma}\Biggl|\sum_{j=1}^Jf_j(+0)\Biggr|^2, \quad & (\text{if }(A,B)\text{ is }\eqref{018}.).
		\end{array}
		\right.
\end{align*}
\end{theorem}

\begin{remark}
When $\gamma < 0$, we cannot get the similar result with Theorem \ref{Theorem of NLSr}, \ref{Theorem of NLSd}, and \ref{Theorem of NLSg}.
For the equation \eqref{NLSr}, the standing wave solution $u(t,x) = e^{i\omega t}Q_{\omega,\gamma}(x)$ was gotten in \cite{Din20}, where $Q_{\omega,\gamma}$ satisfies
\begin{align*}
	- \omega \phi + \partial_x^2 \phi - \frac{\gamma}{|x|^\mu}\phi
		= - |\phi|^4\phi.
\end{align*}
The standing wave solution is time global and has negative energy.
The standing wave solution of \eqref{NLSd} or \eqref{NLSg} can be seen in \cite{AdaCacFinNoj14, FukOhtOza08, Gol22, GooHolWei04}.
\end{remark}

\textbf{Idea of the proof :}
Ogawa--Tsutsumi \cite{OgaTsu91} used a localized virial identity (Proposition \ref{Virial identity}) with a weighted function $\mathscr{X}$ (see \eqref{005} below for the definition).
The function $\mathscr{X}$ is equal to $x^2$ on $\{x : |x| \leq 1\}$ and Ogawa--Tsutsumi collect initial data into $\{x : |x| \leq 1\}$ by the scaling \eqref{002}.
However, the equations \eqref{NLSr}, \eqref{NLSd}, and \eqref{NLSg} are not scale invariant.
So, we replace $\mathscr{X}$ with $\mathscr{X}_R$ and spread a domain (, where the weighted function is equal to $x^2$) by taking sufficiently large $R$.

We note that radial functions have estimate
\begin{align*}
	\|f\|_{L^{p+1}(|x|\geq R)}^{p+1}
		\lesssim R^{-\frac{(d-1)(p-1)}{2}}\|f\|_{L^2(|x|\geq R)}^\frac{p+3}{2}\|\nabla f\|_{L^2(|x|\geq R)}^\frac{p-1}{2}
\end{align*}
for spatial dimension $d$.
Therefore, the weighted function $\mathscr{X}_R$ is often utilized in $d \geq 2$.
In this paper, we apply $\mathscr{X}_R$ in $d = 1$.

\subsection{Organization of the paper}

The organization of the rest of this paper is as follows:
In Section \ref{Sec:Preliminary}, we prepare some notations and tools.
In Section \ref{Sec:Proof}, we give an alternative proof of Theorem \ref{Main theorem}.
In Section \ref{Sec:Application}, we prove a blow-up result of \eqref{NLSr}, \eqref{NLSd}, and \eqref{NLSg} (Theorem \ref{Theorem of NLSr}, \ref{Theorem of NLSd}, and \ref{Theorem of NLSg}) by using the alternative proof.

\section{Preliminary}\label{Sec:Preliminary}

In this section, we define some notations and collect some tools. 

\subsection{Notations and definitions}

For $1 \leq p \leq \infty$, $L^p(\mathbb{R})$ denotes the usual Lebesgue space.
$H^1(\mathbb{R})$ and $W^{s,\infty}(\mathbb{R})$ $(s \in \mathbb{N})$ denote the usual Sobolev spaces.
If a space domain is not specified, then $x$-norm is taken over $R$.
That is, $\|f\|_{L^p} = \|f\|_{L^p(\mathbb{R})}$.

\subsection{Some tools}

\begin{lemma}[Ogawa--Tsutsumi, \cite{OgaTsu91}]
Let $f \in H^1(\mathbb{R})$ and $g \in W^{1,\infty}(\mathbb{R})$ be a real-valued function.
Then, we have
\begin{align*}
	\|fg\|_{L^\infty(|x|\geq R)}
		\leq \|f\|_{L^2(|x|\geq R)}^\frac{1}{2}\left\{2\|g^2\partial_xf\|_{L^2(|x|\geq R)} + \|f\partial_x(g^2)\|_{L^2(|x|\geq R)}\right\}^\frac{1}{2}
\end{align*}
for any $R > 0$.
\end{lemma}

\begin{proposition}[Localized virial identity \Rnum{1}, \cite{OgaTsu91}]\label{Virial identity}
Let $V = 0$.
We assume that $\varphi \in W^{3,\infty}(\mathbb{R})$ has a compact support.
If we define
\begin{align*}
	I_w(t)
		:=\int_{\mathbb{R}}w(x)|u(t,x)|^2dx
\end{align*}
for $w := \int_0^x\varphi(y)dy$ and the solution $u(t)$ to \eqref{NLS0}, then we have
\begin{align*}
	I_w'(t)
		& = 2\text{Im}\int_{\mathbb{R}}w'(x)\overline{u(t,x)}\partial_xu(t,x)dx, \\
	I_w''(t)
		& = 4\int_{\mathbb{R}}w''(x)|\partial_x u(t,x)|^2dx - \frac{4}{3}\int_{\mathbb{R}}w''(x)|u(t,x)|^6dx - \int_{\mathbb{R}}w^{(4)}(x)|u(t,x)|^2dx.
\end{align*}
\end{proposition}

\section{An alternative proof of Theorem \ref{Main theorem}}\label{Sec:Proof}

We define an odd function $\zeta$ as follows :
\begin{align*}
	\zeta(s)
		& :=
		\left\{
		\begin{array}{cl}
		\hspace{-0.2cm}\displaystyle
			2s & \quad (0 \leq |s| \leq 1), \\
		\hspace{-0.2cm}\displaystyle
			2[s-(s-1)^3] & \quad (1 \leq s \leq 1+1/\sqrt{3}), \\
		\hspace{-0.2cm}\displaystyle
			2[s-(s+1)^3] & \quad (-1-1/\sqrt{3} \leq s \leq -1), \\
		\hspace{-0.2cm}\displaystyle
			\zeta'(s) < 0 & \quad (1+1/\sqrt{3} < |s| < 2), \\
		\hspace{-0.2cm}\displaystyle
			0 & \quad (2 \leq |s|).
		\end{array}
		\right.
\end{align*}
For the function $\zeta$, we set the following functions :
\begin{align}
	\mathscr{X}(x)
		& := \int_0^x\zeta(s)ds, \qquad
	\mathscr{X}_R(x)
		:= R^2\mathscr{X}\left(\frac{x}{R}\right). \label{005}
\end{align}

\begin{proposition}\label{Estimate of the virial}
Let $V = 0$.
Let $u_0 \in H^1(\mathbb{R})$.
If the solution $u \in C_t([0,T_\text{max});H_x^1(\mathbb{R}))$ to \eqref{NLS0} with \eqref{IC}, $0 \leq t < T_\text{max}$, and $R > 0$ satisfy
\begin{align}
	\|u(t)\|_{L^2(|x|\geq R)}
		\leq \left(\frac{3}{8}\right)^\frac{1}{4}
		=: a_0, \label{006}
\end{align}
then we have $I_{\mathscr{X}_R}''(t) \leq - 2\~{\eta} := 16 E_0[u_0] + 2\eta$, where
\begin{align*}
	\eta
		:= \frac{4}{3R^2}\left(\sqrt{6} + \frac{\|\zeta''\|_{L^\infty(1+1/\sqrt{3}\leq|x|\leq2)}}{2}\right)^2\|u_0\|_{L^2}^6 + \frac{\|\zeta^{(3)}\|_{L^\infty(1\leq|x|\leq2)}}{2R^2}\|u_0\|_{L^2}^2.
\end{align*}
\end{proposition}

\begin{proof}
Applying Proposition \ref{Virial identity}, we have
\begin{align*}
	I_{\mathscr{X}_R}''(t)
		& = 16E_0[u_0] - \int_{\mathbb{R}}g_{{}_R}(x)^4\left[4|\partial_xu(t,x)|^2 - \frac{4}{3}|u(t,x)|^6\right]dx - \int_{\mathbb{R}}\frac{1}{R^2}\mathscr{X}^{(4)}\left(\frac{x}{R}\right)|u(t,x)|^2dx,
\end{align*}
where $g_{{}_R}$ is defined as
\begin{align*}
	g_{{}_R}(x)
		:= \left\{2 - \mathscr{X}''\left(\frac{x}{R}\right)\right\}^\frac{1}{4}.
\end{align*}
Then, we have
\begin{align*}
	\int_{\mathbb{R}}g_{{}_R}(x)^4|u(t,x)|^6dx
		& = \int_{|x|\geq R}g_{{}_R}(x)^4|u(t,x)|^6dx \\
		& \leq \|u\|_{L^2(|x|\geq R)}^2\|g_{{}_R}u\|_{L^\infty(|x|\geq R)}^4 \\
		& \leq \|u\|_{L^2(|x|\geq R)}^4\left\{2\|g_{{}_R}^2\partial_xu\|_{L^2(|x|\geq R)} + \|u\partial_x(g_{{}_R}^2)\|_{L^2(|x|\geq R)}\right\}^2 \\
		& \leq 8\|u\|_{L^2(|x|\geq R)}^4\|g_{{}_R}^2\partial_xu\|_{L^2(|x|\geq R)}^2 + 2\|u\|_{L^2(|x|\geq R)}^6\|\partial_x(g_{{}_R}^2)\|_{L^\infty(|x|\geq R)}^2.
\end{align*}
By the simple calculation, we have
\begin{align*}
	|\partial_x(g_{_{R}}(x)^2)|
		\left\{
		\begin{array}{ll}
			\hspace{-0.2cm}\displaystyle
				= 0, & \quad (0 \leq |x/R| \leq 1,\ 2 \leq |x/R|), \\
			\hspace{-0.2cm}\displaystyle
				\leq \sqrt{6}/R, & \quad (1 \leq |x/R| \leq 1+1/\sqrt{3}), \\
			\hspace{-0.2cm}\displaystyle
				\leq \frac{1}{2R}\|\zeta''\|_{L^\infty(1+1/\sqrt{3}\leq|x|\leq2)}, & \quad (1+1/\sqrt{3} < |x/R| < 2).
		\end{array}
		\right.
\end{align*}
Therefore, we obtain
\begin{align*}
	I_{\mathscr{X}_R}''(t)
		& \leq 16E_0[u_0] - 4\left\{1 - \frac{8}{3}\|u\|_{L^2(|x|\geq R)}^4\right\}\int_{|x|\geq R}\left\{2 - \mathscr{X}''\left(\frac{x}{R}\right)\right\}|\partial_xu(t,x)|^2dx \notag \\
		& \hspace{0.5cm} + \frac{8}{3R^2}\left(\sqrt{6} + \frac{\|\zeta''\|_{L^\infty(1+1/\sqrt{3}\leq|x|\leq2)}}{2}\right)^2\|u\|_{L^2(|x|\geq R)}^6 + \frac{\|\zeta^{(3)}\|_{L^\infty(1\leq|x|\leq2)}}{R^2}\|u\|_{L^2(|x|\geq R)}^2,
\end{align*}
which completes the proof.
\end{proof}

\begin{proof}[An alternative proof of theorem \ref{Main theorem}]
We consider only positive time.
We assume for contradiction that $u$ exists globally in positive time direction.

We take sufficiently large $R > 0$ satisfying
\begin{gather}
	\~{\eta}
		> 0, \notag \\
	\frac{1}{R}\left(\int_{\mathbb{R}}\mathscr{X}_R(x)|u_0(x)|^2dx\right)^\frac{1}{2}\left(1 + \frac{4}{\~{\eta}}\|\partial_xu_0\|_{L^2}^2\right)^\frac{1}{2}
		\leq \frac{1}{2}a_0, \label{007}
\end{gather}
where $\eta$ and $a_0$ are given in Proposition \ref{Estimate of the virial}.
We note that $\~{\eta} \longrightarrow - 8E_0[u_0] > 0$ as $R \rightarrow \infty$ and
\begin{align*}
	\frac{1}{R^2}\int_{\mathbb{R}}\mathscr{X}_R(x)|u_0(x)|^2dx
		= \int_{\mathbb{R}}\mathscr{X}\left(\frac{x}{R}\right)|u_0(x)|^2dx
		= \int_{\mathbb{R}}\int_0^{x/R}\zeta(s)ds|u_0(x)|^2dx
		\longrightarrow 0
\end{align*}
as $R \rightarrow \infty$ by the dominated convergence theorem.
We prove that $u(t)$ satisfies \eqref{006} for any $0 \leq t < \infty$.
We note that it follows from \eqref{007}, $\~{\eta} > 0$, and $\mathscr{X}_R \geq R^2$ for $|x| \geq R$ that
\begin{align}
	\|u_0\|_{L^2(|x|\geq R)}
		\leq \frac{1}{2}a_0. \label{008}
\end{align}
Here, we define $t_0$ as
\begin{align*}
	t_0
		:= \sup\{t > 0 : \|u(s)\|_{L^2(|x|\geq R)} \leq a_0\ \text{ for any }\ 0 \leq s < t\}.
\end{align*}
By \eqref{008} and the continuity of $\|u(t)\|_{L^2}$, we note $t_0 > 0$.
When $t_0 = \infty$, we get the desired result.
When $0 < t_0 < \infty$, we have $\|u(t_0)\|_{L^2(|x|\geq R)} = a_0$ by the continuity of $\|u(t)\|_{L^2}$.
Then, the solution $u(t)$ satisfies \eqref{006} for any $0 \leq t \leq t_0$, so it follows from Proposition \ref{Estimate of the virial} that
\begin{align*}
	I_{\mathscr{X}_R}''(\tau)
		\leq - 2\~{\eta}
\end{align*}
for any $0 \leq \tau \leq t_0$.
Integrating this inequality over $\tau \in [0,s]$ and over $s \in [0,t]$,
\begin{align}
	I_{\mathscr{X}_R}(t)
		\leq I_{\mathscr{X}_R}(0) + I_{\mathscr{X}_R}'(0)t - \~{\eta} t^2 \label{009}
\end{align}
Combining \eqref{009} and $\~{\eta} > 0$, we have
\begin{align}
	I_{\mathscr{X}_R}(t)
		\leq I_{\mathscr{X}_R}(0) - \~{\eta}\left\{t - \frac{1}{2\~{\eta}}I_{\mathscr{X}_R}'(0)\right\}^2 + \frac{1}{4\~{\eta}}I_{\mathscr{X}_R}'(0)^2
		\leq I_{\mathscr{X}_R}(0) + \frac{1}{\~{\eta}}\|u_0 \mathscr{X}_R'\|_{L^2}^2\|\partial_x u_0\|_{L^2}^2 \label{010}
\end{align}
for any $0 \leq t \leq t_0$.
$\mathscr{X}_R \geq R^2$ $(|x| \geq R)$, \eqref{010}, $(\mathscr{X}_R')^2 \leq 4\mathscr{X}_R$, and \eqref{007} tell us
\begin{align*}
	\|u(t)\|_{L^2(|x|\geq R)}
		\leq \frac{1}{R}I_{\mathscr{X}_R}(t)^\frac{1}{2}
		\leq \frac{1}{R}I_{\mathscr{X}_R}(0)^\frac{1}{2}\left(1 + \frac{4}{\~{\eta}}\|\partial_xu_0\|_{L^2}^2\right)^\frac{1}{2}
		\leq \frac{1}{2}a_0
\end{align*}
for any $0 \leq t \leq t_0$.
However, this is contradiction.

When $t_0 = \infty$, \eqref{009} deduces that $I_{\mathscr{X}_R}(t) < 0$ in finite time, which is contradiction.
Therefore, the solution $u$ to \eqref{NLS0} with \eqref{IC} blows up.
\end{proof}

\section{Applications}\label{Sec:Application}

To prove Theorem \ref{Theorem of NLSr}, we use the following localized virial identity.

\begin{proposition}[Localized virial identity \Rnum{2}, \cite{Din21, TaoVisZha07}]\label{Virial identity with potential}
Let $V = \frac{\gamma}{|x|^\mu}$ and $0 < \mu < 1$.
We assume that $\varphi \in W^{3,\infty}(\mathbb{R})$ has a compact support.
If we define
\begin{align*}
	I_{\gamma,w}(t)
		:=\int_{\mathbb{R}}w(x)|u(t,x)|^2dx
\end{align*}
for $w := \int_0^x\varphi(y)dy$ and the solution $u(t)$ to \eqref{NLSr}, then we have
\begin{align*}
	I_{\gamma,w}'(t)
		& = 2\text{Im}\int_{\mathbb{R}}w'(x)\overline{u(t,x)}\partial_xu(t,x)dx, \\
	I_{\gamma,w}''(t)
		& = 4\int_{\mathbb{R}}w''(x)|\partial_x u(t,x)|^2dx - \frac{4}{3}\int_{\mathbb{R}}w''(x)|u(t,x)|^6dx \\
		& \hspace{4.0cm} - \int_{\mathbb{R}}w^{(4)}(x)|u(t,x)|^2dx + 2\mu\int_{\mathbb{R}}\frac{w'(x)}{x}\cdot\frac{\gamma}{|x|^\mu}|u(t,x)|^2dx.
\end{align*}
\end{proposition}

Applying Proposition \ref{Virial identity with potential} with the weighted function $\mathscr{X}_R$, we have
\begin{align*}
	I_{\gamma,\mathscr{X}_R}''(t)
		\leq 16E_\gamma[u_0] + 2\eta + 2\int_{\mathbb{R}}\left\{\mu\frac{R}{x}\mathscr{X}'\left(\frac{x}{R}\right) - 4\right\}\frac{\gamma}{|x|^\mu}|u(t,x)|^2dx
\end{align*}
by the same argument with Proposition \ref{Estimate of the virial}, where $\eta$ is given in Proposition \ref{Estimate of the virial}.
It follows from $\mu\frac{R}{|x|}\mathscr{X}'\left(\frac{x}{R}\right) - 4 \leq 0$ that
\begin{align*}
	I_{\gamma,\mathscr{X}_R}''(t)
		\leq 16E_\gamma[u_0] + 2\eta.
\end{align*}
The rest of the proof of Theorem \ref{Theorem of NLSr} is the same with Theorem \ref{Main theorem}.

We turn to the nonlinear Schr\"odinger equation with the delta potential.
To prove Theorem \ref{Theorem of NLSd}, we use the following localized virial identity.

\begin{proposition}[Localized virial identity \Rnum{3}, \cite{BanVis16, IkeInu17}]\label{Virial identity with delta potential}
Let $V = \gamma\delta$.
We assume that $\varphi \in W^{3,\infty}(\mathbb{R})$ has a compact support and satisfies $\varphi(0) = 0$.
If we define
\begin{align*}
	I_{\delta,w}(t)
		:=\int_{\mathbb{R}}w(x)|u(t,x)|^2dx
\end{align*}
for $w := \int_0^x\varphi(y)dy$ and the solution $u(t)$ to \eqref{NLSd}, then we have
\begin{align*}
	I_{\delta,w}'(t)
		& = 2\text{Im}\int_{\mathbb{R}}w'(x)\overline{u(t,x)}\partial_xu(t,x)dx, \\
	I_{\delta,w}''(t)
		& = 4\int_{\mathbb{R}}w''(x)|\partial_x u(t,x)|^2dx - \frac{4}{3}\int_{\mathbb{R}}w''(x)|u(t,x)|^6dx \\
		& \hspace{4.0cm} - \int_{\mathbb{R}}w^{(4)}(x)|u(t,x)|^2dx + 2\gamma w''(0)|u(t,0)|^2.
\end{align*}
\end{proposition}

Applying Proposition \ref{Virial identity with delta potential} with the weighted function $\mathscr{X}_R$, we have
\begin{align*}
	I_{\delta,\mathscr{X}_R}''(t)
		\leq 16E_{\gamma\delta}[u_0] + 2\eta - 4\gamma|u(t,0)|^2
\end{align*}
by the same argument with Proposition \ref{Estimate of the virial}, where $\eta$ is given in Proposition \ref{Estimate of the virial}.
It follows from $- 4\gamma|u(t,0)|^2 \leq 0$ that
\begin{align*}
	I_{\delta,\mathscr{X}_R}''(t)
		\leq 16E_{\gamma\delta}[u_0] + 2\eta.
\end{align*}
The rest of the proof of Theorem \ref{Theorem of NLSd} is the same with Theorem \ref{Main theorem}.

To prove Theorem \ref{Theorem of NLSg}, we use the following localized virial identity.

\begin{proposition}[Localized virial identity \Rnum{4}, \cite{Gol22, GolOht20}]\label{Virial identity on graph}
We assume that $\varphi \in W^{3,\infty}(0,\infty)$ has compact support and satisfies $\varphi(0) = 0$.
If we define
\begin{align*}
	I_{\mathcal{G},w}(t)
		:= \int_\mathcal{G}w(x)|\bm{u}(t,x)|^2dx.
\end{align*}
for $w = \int_0^x\varphi(y)dy$ and the solution $\bm{u}$ to \eqref{NLSg}, then we have
\begin{align*}
	I_{\mathcal{G},w}'(t)
		& = 2\text{Im}\int_\mathcal{G}w'(x)\overline{\bm{u}(t,x)}\partial_x\bm{u}(t,x)dx, \\
	I_{\mathcal{G},w}''(t)
		& = 4\int_\mathcal{G}w''(x)|\partial_x\bm{u}(t,x)|^2dx - \frac{4}{3}\int_\mathcal{G}w''(x)|\bm{u}(t,x)|^6dx \\
		& \hspace{4.0cm} - \int_\mathcal{G}w^{(4)}(x)|\bm{u}(t,x)|^2dx + 2w''(0)P(\bm{u}(t)),
\end{align*}
where $P$ is defined in Theorem \ref{Theorem of NLSg}.
\end{proposition}

Applying Proposition \ref{Virial identity on graph} with the weighted function $\mathscr{X}_R$, we have
\begin{align*}
	I_{\mathcal{G},\mathscr{X}_R}''(t)
		\leq 16E_\mathcal{G}[\bm{u_0}] + 2\eta_{{}_\mathcal{G}} - 4P(u(t))
\end{align*}
by the same argument with Proposition \ref{Estimate of the virial}, where
\begin{align*}
	\eta_{{}_\mathcal{G}}
		:= \frac{4}{3R^2}\left(\sqrt{6} + \frac{\|\zeta''\|_{L^\infty(1+1/\sqrt{3}\leq|x|\leq2)}}{2}\right)^2\|\bm{u}\|_{L^2(\mathcal{G})}^6 + \frac{\|\zeta^{(3)}\|_{L^\infty(1\leq|x|\leq2)}}{2R^2}\|\bm{u}\|_{L^2(\mathcal{G})}^2,
\end{align*}
It follows from $- 4P(u(t)) \leq 0$ that
\begin{align*}
	I_{\mathcal{G},\mathscr{X}_R}''(t)
		\leq 16E_\mathcal{G}[\bm{u_0}] + 2\eta_{{}_\mathcal{G}}.
\end{align*}
The rest of the proof of Theorem \ref{Theorem of NLSg} is the same with Theorem \ref{Main theorem}.

\subsection*{Acknowledgements}
The first author is supported by Foundation of Research Fellows, The Mathematical Society of Japan.
The second author is supported by JSPS KAKENHI Grant Number JP18H01132, JP19K14581, and JST CREST Grant Number JPMJCR1913.
The third author is supported by JSPS KAKENHI Grant Number JP19H00644 and JP20K03671.

\end{document}